\documentclass[a4paper,12pt]{article}
\usepackage{amsmath,amssymb,amsthm}
\usepackage{fullpage,hyperref,enumitem}

\theoremstyle{plain}
\newtheorem{theorem}{Theorem}[section]
\newtheorem{lemma}[theorem]{Lemma}
\newtheorem{proposition}[theorem]{Proposition}
\newtheorem{corollary}[theorem]{Corollary}

\theoremstyle{remark}
\newtheorem{remark}[theorem]{Remark}

\theoremstyle{definition}
\newtheorem{definition}[theorem]{Definition}

\setlist{nosep} 

\def\beq{\begin{equation}}
\def\eeq{\end{equation}}

\DeclareMathOperator{\diam}{diam}
\DeclareMathOperator{\Geom}{Geom}
\DeclareMathOperator{\bE}{\mathbb{E}}
\DeclareMathOperator{\bP}{\mathbb{P}}

\def\eps{{\varepsilon}}

\def\bN{{\mathbb{N}}}
\def\bR{{\mathbb{R}}}
\def\bW{{\mathbb{W}}}
\def\bZ{{\mathbb{Z}}}

\def\cA{{\mathcal{A}}}
\def\cB{{\mathcal{B}}}

\def\cR{{\mathcal{R}}}

\title{Rates in almost sure invariance principle for dynamical systems with some hyperbolicity}

\author{Alexey Korepanov \\
{\small University of Exeter, UK}}

\date{27 April 2017 (updated 10 July 2018)}

\begin{document}

\maketitle

\begin{abstract}
  We prove the almost sure invariance principle 
  with rate $o(n^{\varepsilon})$ for every $\varepsilon > 0$ for 
  H\"older continuous observables
  on nonuniformly expanding and nonuniformly hyperbolic transformations
  with exponential tails. Examples include Gibbs-Markov maps with big images,
  Axiom A diffeomorphisms, dispersing billiards and
  a class of logistic and H\'enon maps.
  The best previously proved rate is
  $O(n^{1/4} (\log n)^{1/2} (\log \log n)^{1/4})$.
  
  As a part of our method, we show that nonuniformly expanding transformations
  are factors of Markov shifts with simple structure and 
  natural metric (similar to the classical Young towers). The factor map is
  Lipschitz continuous and probability measure preserving.
  For this we do not require the exponential tails.
\end{abstract}


\maketitle

\section{Introduction}

\begin{definition}
  We say that a random process \(X_0, X_1, \ldots\) satisfies the 
  \emph{Almost Sure Invariance Principle} (ASIP) with rate, say \(o(n^{\eps})\)
  with \(\eps \in (0,1/2)\), if without changing the distribution,
  \(\{X_n, n \geq 0\}\) can be redefined on a new probability space
  with a Brownian motion \(W_t\) such that
  \[
    X_n = W_n + o(n^{\eps})
    \qquad \text{almost surely}
    .
  \]
\end{definition}

The ASIP is a strong statistical property. It implies directly
the functional central limit theorem, the functional law of iterated logarithm
and other statistical laws, see Philipp and Stout~\cite[Chapter 1]{PS75}.
The rate in the ASIP has additional powerful implications, see
Berkes, Liu and Wu~\cite{BLW14} and references therein.

Suppose that \(T \colon \Lambda \to \Lambda\) is a 
nonuniformly expanding or nonuniformly hyperbolic transformation
as in Young~\cite{Y98,Y99} with \emph{exponential tails}
(see Section~\ref{sec:YT}),
such as Gibbs-Markov maps with big images,  Axiom A diffeomorphisms,
dispersing billiards and a class of logistic and H\'enon maps.

Suppose that \(\nu\) is the unique \(T\)-invariant ergodic physical measure,
\(v \colon \Lambda \to \bR\) is a H\"older continuous observable
with \(\int v \, d\nu = 0\) and \(v_n = \sum_{k=0}^{n-1} v \circ T^k\).
Then \(v_n, n \geq 0\) is a random process with stationary increments on the probability
space \((\Lambda, \nu)\).

We prove that  \(v_n\) satisfies the ASIP with rate \(o(n^{\eps})\) for every \(\eps > 0\).
Our results strongly improve the best previously available rates.

\begin{remark}
  Our analysis is restricted to discrete time
  \(\bR\)-valued processes.
  The ASIP with good rates for flows and \(\bR^d\)-valued processes
  with dependent increments, such as those in dynamical systems,
  is an important open problem.
\end{remark}

\begin{remark}
  We only consider processes with bounded increments. This is automatic
  for dynamical systems with H\"older continuous observables as above.
\end{remark}

The ASIP has been introduced by Strassen~\cite{S64,S67}, proved for processes
with independent increments and martingales using the Skorokhod embedding.
Approximations with martingales turned out to be very robust, see
Philipp and Stout~\cite{PS75};
they have been used to prove the ASIP for various dynamical systems
\cite{CB01,CM15,DP84,FMT03,HK82,MN05,MR12},
including the nonuniformly expanding and nonuniformly hyperbolic maps in
\cite{MN05}.

A downside of the martingale method is that the best rate in the ASIP which
the Skorokhod embedding can produce is 
\(O(n^{1/4} (\log n)^{1/2} (\log \log n)^{1/4})\), see Kiefer \cite{K69}.
For nonuniformly expanding and nonuniformly hyperbolic systems,
this rate has been achieved by Cuny and Merlev\`ede \cite{CM15}
and Korepanov, Kosloff and Melbourne~\cite{KKM16mart}.

For processes with independent and identically distributed increments,
Koml\'os, Major and Tusn\'ady in their celebrated work~\cite{KMT75} proved the ASIP with a
much better rate \(O(\log n)\), which is in fact unimprovable.
Their proof is based on the so-called Hungarian construction
and uses the quantile transform rather than the Skorokhod embedding.

For processes with dependent increments, it is also possible 
to prove the ASIP without relying on the Skorokhod embedding,
but getting good rates proved to be challenging. For instance, 
Gou\"ezel~\cite{G10} used blocking techniques to construct an approximation
with a process with independent increments, for which the ASIP
with the optimal rate \(O(\log n)\) is known. However, an efficient control of the 
approximation error is tricky, and the best rate he could reach was 
\(o(n^{1/4+\eps})\) for every \(\eps > 0\), roughly the same as in the martingale method.
For different reasons, \(o(n^{1/4+\eps})\) was not surpassed by various other
methods~\cite{BP79,KP80,MN09,MR12}.

In the dependent setting, the rate \(O(n^{1/4} (\log n)^{1/2} (\log \log n)^{1/4})\)
was unbeaten until very recently.
First, Berkes, Liu and Wu \cite{BLW14} proved the ASIP with rate 
\(o(n^\eps)\) for every \(\eps > 0\) for processes generated by a Bernoulli shift:
\[
  X_n = \sum_{k=0}^{n-1} \psi(\ldots, \xi_{k-1}, \xi_k, \xi_{k+1}, \ldots)
  ,
\]
where \(\{\xi_k\}\) is a sequence of independent identically distributed
random variables and \(\psi\) is a sufficiently nice function.
Their result is based on an insightful approximation by
a process with independent increments and
a Koml\'os-Major-Tusn\'ady type result for processes with independent
but not identically distributed increments by Sakhanenko~\cite{S06}.
Soon after, Merlev\`ede and Rio~\cite{MR15} obtained the rate \(O(\log n)\)
for Harris recurrent geometrically ergodic Markov chains, strongly using the
Markovian structure and in particular the regeneration technique.

The result of \cite{BLW14} readily covers some smooth dynamical systems such as
the doubling map \(x \mapsto 2x \pmod{1}\), whose natural symbolic coding
is a Bernoulli shift. But such systems are rare: for instance, they
do not include smooth perturbations of the doubling map.

In the present work we extend the result of~\cite{BLW14} to
a large class of widely studied dynamical systems.
Our strategy is to construct a semiconjugacy between a dynamical system
in question and a Bernoulli shift. The semiconjugacy preserves the probability
measure and sufficient structure for verification of assumptions of~\cite{BLW14}.

\begin{remark}  
  The historical overview above focuses on what is immediately relevant
  to our goals, without any attempt to describe the vast literature
  on the ASIP. For a thorough description of rates related results,
  see~\cite{BLW14} and the review by Zaitsev~\cite{Z13}.
\end{remark}

\begin{remark}
  Chernov and Haskell \cite{CH96} prove the \emph{Bernoulli property} for 
  K-mixing nonuniformly hyperbolic maps.
  That is, such maps are measure-theoretically isomorphic to Bernoulli shifts.
  They remark that even though the Bernoulli property is a characterization
  of extreme chaotic behavior, it is not helpful in proving statistical
  properties like the central limit theorem. This is because a measure-theoretic
  isomorphism alone does not have to preserve any useful information about the
  structure of the space, such as metric or coordinates.
  
  In contrast, we build a semiconjugacy to a Bernoulli shift which preserves
  enough information to prove the ASIP.
\end{remark}

\begin{remark}
  As an essential part of our proof, for a nonuniformly expanding
  dynamical system we construct an extension which is a renewal Markov shift,
  so that the factor map is Lipschitz with respect to a natural metric.
  Our construction is inspired by the coupling lemma for dispersing
  billiards, as it appears in Chernov and Markarian~\cite[Lemma~7.24]{CM06}.

  After circulating the first version of this paper, the author has been
  made aware of the work by Zweim\"uller~\cite{Z09}, where he shows that
  nonuniformly expanding dynamical systems are similar to renewal Markov shifts.

  Two dynamical systems are similar if they are factors of a common
  extension. All probability measure preserving systems are trivially
  similar, but in infinite ergodic theory the similarity is a highly
  nontrivial relation. The focus of~\cite{Z09} is on infinite measure
  preserving systems.

  Our construction is remarkably similar to the one in~\cite{Z09},
  although we draw rather different conclusions:
  we make observations which allow us to treat probability measure preserving
  systems.
\end{remark}

\section{Statement of the result}
\label{sec:YT}

We use notation \(\bN = \{1,2,\ldots\}\) and \(\bN_0 = \{0,1,\ldots\}\).
All functions, subsets and partitions are assumed to be measurable.
When we work with metric spaces, the default sigma algebra is Borel,
and for finite and countable spaces the sigma algebra is discrete.

Let \((\Lambda,d_\Lambda)\) be a bounded metric space and 
\(T \colon \Lambda \to \Lambda\) be a transformation.
Let \(Y\) be a subset of \(\Lambda\) and \(m\) be 
a probability measure on \(Y\). Let \(\alpha\) be an at most countable
partition of \(Y\) (modulo a zero measure set) such that \(m(a) > 0\)
for all \(a \in \alpha\).

Let \(\tau \colon Y \to \bN\) be an integrable function
which is constant on each \(a \in \alpha\) with value \(\tau(a)\)
such that \(T^{\tau(a)} (y) \in Y\) for every \(y \in a\), \(a \in \alpha\).
Let \(F \colon Y \to Y\), \(F(y) = T^{\tau(y)} (y)\)
be the induced map.

We assume that for each \(a \in \alpha\), the map 
\(F\) restricts to a (measure-theoretic)
bijection from \(a\) to \(Y\).
Further, there are constants
\(0 < \eta \leq 1\), \(\lambda > 1\) and \(K, K_\tau \geq 1\)
such that for all \(a \in \alpha\) and \(x,y \in a\):
\begin{itemize}
  \item \(d_\Lambda(F(x), F(y)) \geq \lambda d(x,y)\),
  \item \(d_\Lambda(T^\ell (x), T^\ell (y)) 
    \leq K_\tau \, d_\Lambda(F(x), F(y))\)
    for all \(0 \leq \ell < \tau(a)\),
  \item the restriction 
    \(F \colon a \to Y\) is nonsingular and its inverse Jacobian
    \(\zeta_a = \frac{dm}{dm \, \circ \, F}\) satisfies
      \begin{equation}
        \label{eq:dist}
        |\log \zeta_a(x) - \log \zeta_a(y)| 
        \leq K d_\Lambda^\eta(F(x), F(y))
        .
      \end{equation}
\end{itemize}

Finally, we assume that the induced map $F \colon Y \to Y$ 
allows a non-pathological coding by elements of \(\alpha\).
We require that the set
\[
  \{ (a_0, a_1, \ldots) \in \alpha^{\bN_0} :
  \text{there exists }y \in Y \text{ with } F^k(y) \in a_k \text{ for all } k  \}
\]
is measurable in \(\alpha^{\bN_0}\) (in the product topology with Borel sigma algebra).

We say that \(T \colon \Lambda \to \Lambda\) as above is a
\emph{nonuniformly expanding} map. We say that it has \emph{exponential
(return time) tails,} if \(\int_Y e^{\beta \tau} \, dm < \infty\)
with some \(\beta > 0\).

It is standard \cite{Y99} that there is a unique 
\(T\)-invariant ergodic probability measure on \(\Lambda\),
with respect to which \(m\) is absolutely continuous.
We denote this measure by \(\nu\).

For on observable \(v \colon \Lambda \to \bR\), denote
\[
  |v|_\infty = \sup_{x \in \Lambda} |v(x)|
  , \quad
  |v|_\eta = \sup_{x \neq y \in \Lambda} \frac{|v(x) - v(y)|}{d^\eta(x,y)}
  \quad \text{and} \quad
  \|v\|_\eta = |v|_\infty + |v|_\eta
  .
\]
We say that \(v\) is \emph{centered}, if \(\int v \, d\nu = 0\),
and that \(v\) is H\"older, if \(\|v\|_\eta < \infty\).

Our main result is:

\begin{theorem}
  \label{thm:dura}
  Suppose that there exists \(\beta > 0\) such that 
  \(\int_Y e^{\beta \tau} \, dm < \infty\).
  If \(v \colon \Lambda \to \bR\) is a H\"older centered observable,
  then the process \(v_n = \sum_{k=0}^{n-1} v \circ T^k\),
  defined on the probability space \((\Lambda, \nu)\), satisfies the ASIP
  with rate \(o(n^{\eps})\) for every \(\eps > 0\).
\end{theorem}

\begin{remark}
  For maps which are naturally a Bernoulli shift, such as the doubling map, 
  Theorem~\ref{thm:dura} follows directly from~\cite{BLW14}.
  Our result is new for smooth perturbations of the doubling map
  and, for example, for:
  \begin{itemize}
    \item smooth expanding circle maps,
    \item Gibbs-Markov maps with big images,
    \item unimodal maps such as logistic with Collet-Eckmann parameters~\cite{BLS03}.
  \end{itemize}
\end{remark}

\begin{remark}
  In nonuniformly hyperbolic maps with exponential tails
  and uniform contraction along stable leaves, as in Young \cite{Y98},
  H\"older observables reduce to H\"older observables on
  a nonuniformly expanding quotient system through a bounded coboundary.
  A detailed exposition can be found in~\cite[Section~5]{KKM16mart}.
  Thus Theorem~\ref{thm:dura} implies the ASIP with rate \(o(n^{\eps})\) for
  every \(\eps>0\) for maps such as:
  \begin{itemize}
    \item Anosov and Axiom A diffeomorphisms,
    \item dispersing billiards,
    \item H\'enon maps with Benedicks-Carleson parameters,
    \item Lozi maps.
  \end{itemize}
\end{remark}

The paper is organized as follows: in Section~\ref{sec:BYT} we introduce
the notion of \emph{Markov Young towers} and state Theorem~\ref{thm:joke}
which establishes a semiconjugacy between \(T \colon \Lambda \to \Lambda\) and
a Markov Young tower.
Theorem~\ref{thm:joke} is proved in Section~\ref{sec:joke}.
We prove Theorem~\ref{thm:dura} in Section~\ref{sec:dura}.

\section{Markov Young towers}
\label{sec:BYT}

Suppose that
\begin{itemize}
  \item \((\cA, \bP_\cA)\) is a finite or countable probability space,
  \item \(h_\cA \colon \cA \to \bN\) is an integrable function,
  \item \(0 < \xi < 1\) is a constant.
\end{itemize}

Define a probability space \((X, \bP_X) = (\cA^\bN, \bP_\cA^\bN)\) and let
\(f_X \colon X \to X\) be the left shift,
\[
  f_X(a_0, a_1, \ldots) = (a_1, a_2, \ldots)
  .
\]
Define \(h \colon X \to \bN\), \(h(a_0, a_1, \ldots) = h_\cA(a_0)\).
Let \(f \colon \Delta \to \Delta\) be a suspension over \(f_X \colon X \to X\) with
a roof function \(h\), i.e.\ 
\beq \begin{aligned}
  \label{eq:susp}
  \Delta & = \{(x,\ell) \in X \times \bZ
    : 0 \leq \ell < h(x)\}
  \\
  f(x,\ell) & = \begin{cases} (x, \ell + 1), & \ell < h(x) - 1, \\
    (f_X (x), 0), & \ell = h(x) - 1 \end{cases}
  .
\end{aligned} \eeq

Define a distance \(d\) on \(X\) by \(d(x,y) = \xi^{s(x,y)}\),
where \(s \colon X \times X \to \bN_0\) is the separation time,
\[
  s((a_0, a_1, \ldots), (b_0, b_1, \ldots) )
  = \inf \{ j \geq 0 : a_j \neq b_j \}
  .
\]
Let \(d\) also denote the natural compatible distance on 
\(\Delta\):
\beq
  \label{eq:cdis}
  d((x,k),(y,j)) = \begin{cases}
  1, & k \neq j \\
  d(x,y), & k=j
  \end{cases}.
\eeq

Let \(\bar{h} = \int h \, dm\). Let \(\bP\) be the probability measure
on \(\Delta\) given by \(\bP(A \times \{\ell\}) = \bar{h}^{-1} m(A)\) for all
\(\ell \geq 0\) and \(A \subset \{y \in Y : h(y) \geq \ell + 1\}\).
Note that \(\bP\) is \(f\)-invariant.

Let \(\Delta_k = \{(y,\ell) \in \Delta : \ell = k\}\).
Then \(X\) is naturally identified with \(\Delta_0\), which we refer to as
the \emph{base} of the suspension, and \(\bP_X\), \(f_X\)
have their counterparts on \(\Delta_0\), which we also denote
\(\bP_X\), \(f_X\).

\begin{definition}
  We call the map \(f \colon \Delta \to \Delta\)
  as above a (non-invertible) \emph{Markov Young tower}.
\end{definition}

\begin{remark}
  \label{rmk:defBYT}
  To define a Markov Young tower, we need an at most countable probability space
  \((\cA, \bP_\cA)\), an integrable function \(h_\cA \colon \cA \to \bN\)
  and a constant \(0 < \xi < 1\). Further we always use notation for 
  Markov Young towers as above, i.e.\  with the
  symbols \(f\), \(\Delta\), \(\cA\), \(\bP_\cA\), \(X\), \(\bP_X\),
  \(f_X\), \(h\), \(\bP\), \(d\), \(\xi\).
\end{remark}

\begin{remark}
  Similar to the classical Young towers, our Markov Young towers are very simple
  objects on their own, studied by various people under different names.
  We chose the term ``Markov Young tower'' because for both,
  the key property is not just their own structure but
  the relation to a large class of dynamical systems (see Theorem~\ref{thm:joke}
  below). It allowed Young~\cite{Y98} to prove the exponential decay of
  correlations for dispersing billiards among other maps, and it is
  an essential ingredient in the proof of our main result.
\end{remark}

Our key technical result is:

\begin{theorem}
  \label{thm:joke}
  Suppose that \(T \colon \Lambda \to \Lambda\) is a nonuniformly expanding map.
  Then there exists a Markov Young tower \(f \colon \Delta \to \Delta\) and
  a map \(\pi \colon \Delta \to \Lambda\), defined \(\bP\)-almost everywhere,
  such that
  \begin{itemize}
    \item \(\pi\) is Lipschitz:
      \[
        d_\Lambda(\pi(x), \pi(y))
        \leq C_\Lambda d(x,y)
        ,
      \]
      where \(C_\Lambda = \lambda K_\tau \, \diam \Lambda\),
    \item \(\pi\) is a semiconjugacy: \(\bP\)-almost surely, \(T \circ \pi = \pi \circ f\),
    \item \(\pi\) preserves the probability measures: \(\pi_* \bP_X = m\) and \(\pi_* \bP = \nu\).
  \end{itemize}

  In addition, moments of \(h\) are closely related to those of \(\tau\):
  \begin{itemize}
    \item (Weak polynomial moments)
      If there exist \(C_\tau > 0\) and \(\beta > 1\) such that
      \(m(\tau \geq \ell) \leq C_\tau \ell^{-\beta}\) for all \(\ell \geq 1\),
      then \(\bP_X(h \geq \ell) \leq C \ell^{-\beta}\) for all \(\ell \geq 1\),
      where the constant \(C\) continuously depends on \(C_\tau\), \(\beta\), \(\lambda\),
      \(K\) and \(\eta\).
    \item (Strong polynomial moments)
      If there exist constants \(C_\tau > 0\) and \(\beta > 1\) such that
      \(\int \tau^\beta \, dm \leq C_\tau \),
      then \(\int h^\beta \, d\bP_X \leq C \),
      where the constant \(C\) continuously depends on \(C_\tau\), \(\beta\), \(\lambda\),
      \(K\) and \(\eta\).
    \item (Exponential and stretched exponential moments)
      If there exist constants \(C_\tau > 0\), \(\beta > 0\) and \(\gamma \in (0,1]\)
      such that \(\int e^{\beta \tau^\gamma} \, dm \leq C_\tau\), then
      \(\int e^{\beta' h^\gamma} \, d\bP_X \leq C\), where the constants \(\beta' \in (0,\beta]\)
      and \(C > 0\) depend continuously on \(C_\tau\), \(\beta\), \(\gamma\), \(\lambda\),
      \(K\) and \(\eta\).
    \item (Exactly exponential moments)
      If \(\int e^{\beta \tau} \, dm < \infty\) for some \(\beta > 0\),
      then \(f \colon \Delta \to \Delta\) can be constructed so that 
      \[
        \bP_X(h = n) = 
        \begin{cases}
          \theta (1-\theta)^{-1} (1-\theta)^{n/N}, & n \in \{N, 2N, 3N, \dots\} \\
          0, & \text{else}
        \end{cases}
      \]
      with some \(0 < \theta < 1\) and \(N \geq 1\).
  \end{itemize}
\end{theorem}

\begin{remark}
  The exact exponential moments in Theorem~\ref{thm:joke} allow us to represent in a natural
  way \(f \colon \Delta \to \Delta\) as a factor of a Bernoulli
  shift and use \cite{BLW14} to prove the ASIP, see Section~\ref{sec:dura}.
  Our results are limited to \(\int e^{\beta \tau} \, dm < \infty\), 
  because without the exact exponential moments such a representation does not work.
\end{remark}

\section{Proof of Theorem~\ref{thm:joke}}
\label{sec:joke}

For \(v \colon \Delta \to \bR\) and \(\eta \in (0,1]\), denote
\[
  |v|_\infty = \sup_{x \in \Delta} |v(x)|
  , \quad
  |v|_\eta = \sup_{x \neq y \in \Delta} \frac{|v(x) - v(y)|}{d^\eta(x,y)}
  \quad \text{and} \quad
  \|v\|_\eta = |v|_\infty + |v|_\eta
  .
\]

\subsection{Construction of Markov Young tower}

We define \(\cA\) as the set of all finite words in the alphabet \(\alpha\)
(not including the empty word). 
For \(w = a_0 \ldots a_{n-1} \in \cA\) we define
\[
  |w| = n
  \qquad \text{and} \qquad
  h_\cA(w) = \tau(a_0) + \cdots + \tau(a_{n-1})
  .
\]
Let
\[
  Y_w = 
  \{ y \in Y : T^k \in a_{k} 
  \text{ for all } 0 \leq k \leq n-1 \}
  .
\]

We use the measure \(\bP_\cA\) from the following lemma:
\begin{lemma}
  \label{lemma:disi}
  There exists a probability measure \(\bP_\cA\) on \(\cA\)
  and a disintegration
  \(m = \sum_{w \in \cA} \bP_\cA(w) m_w\), 
  where \(m_w\) are probability measures on \(Y\), such that
  for every \(w \in \cA\),
  \begin{itemize}
    \item \(m_w\) is supported on \(Y_w\),
    \item \((T^{h_\cA(w)})_* m_w = m\).
  \end{itemize}
  In addition,
  \begin{itemize}
    \item
      If there exist \(C_\tau > 0\) and \(\beta > 1\) such that
      \(m(\tau \geq \ell) \leq C_\tau \ell^{-\beta}\) for all \(\ell \geq 1\),
      then \(\bP_\cA(h_\cA \geq \ell) \leq C \ell^{-\beta}\) for all \(\ell \geq 1\),
      where the constant \(C\) continuously depends on \(C_\tau\), \(\beta\), \(\lambda\),
      \(K\) and \(\eta\).
    \item
      If there exist constants \(C_\tau > 0\) and \(\beta > 1\) such that
      \(\int \tau^\beta \, dm \leq C_\tau \),
      then \(\int h_\cA^\beta \, d\bP_\cA \leq C \),
      where the constant \(C\) continuously depends on \(C_\tau\), \(\beta\), \(\lambda\),
      \(K\) and \(\eta\).
    \item
      If there exist constants \(C_\tau > 0\), \(\beta > 0\) and \(\gamma \in (0,1]\)
      such that \(\int e^{\beta \tau^\gamma} \, dm \leq C_\tau\), then
      \(\int e^{\beta' h_\cA^\gamma} \, d\bP_\cA \leq C\), where the constants \(\beta' \in (0,\beta]\)
      and \(C > 0\) depend continuously on \(C_\tau\), \(\beta\), \(\gamma\), \(\lambda\),
      \(K\) and \(\eta\).
  \end{itemize}
\end{lemma}

\begin{remark}
  Our Lemma~\ref{lemma:disi} corresponds to \cite[Theorem~2]{Z09},
  where the disintegration of \(m\) is called a \emph{regenerative partition
  of unity.}
  For the ease of citation and explicit tail estimates, we refer to \cite{K17}.
\end{remark}

\begin{proof}[Proof of Lemma~\ref{lemma:disi}]
  Such a decomposition is constructed in 
  \cite[Section~4]{K17}. It is implicit in \cite{K17}
  that \(m_w\) is supported on \(Y_w\).

  We remark that in~\cite{K17}, the set \(\cA\) contains the empty word, while here
  we do not allow it. This, however, does not cause problems, because if \(w\) is
  the empty word, then \(\bP_\cA(w)\) is uniformly bounded away from \(1\)
  and \(m_w = m\). Thus the decomposition with the empty word translates to
  one without, with the same moment bounds.
\end{proof}

For the exactly exponential moments in Theorem~\ref{thm:joke}, we obtain a special
version of Lemma~\ref{lemma:disi}:

\begin{lemma}
  \label{lemma:ccnnjj}
  Suppose that \(\int e^{\beta \tau} \, dm < \infty\) with some \(\beta > 0\).
  Then the measure \(\bP_\cA\) in Lemma~\ref{lemma:disi} can be chosen so that 
  \[
    \bP_\cA(h_\cA = \ell) = 
    \begin{cases}
      \theta^{-1} (1-\theta) \theta^{\ell/N} , & \ell \in N \bN \\
      0, & \text{else}
    \end{cases}
  \]
  with some \(N \in \bN\) and \(0 < \theta < 1\).
\end{lemma}

Our proof of Lemma~\ref{lemma:ccnnjj} uses a rather delicate technical
adaptation of the argument in \cite[Section 4]{KKM16exp}. It is carried
out in Appendix~\ref{sec:kkk}.

Let \(\bP_\cA\) and \(\{m_w\}\) be as in Lemmas~\ref{lemma:disi} or ~\ref{lemma:ccnnjj}.
Let \(\xi = \lambda^{-1}\).
According  to Remark~\ref{rmk:defBYT}, \(\cA\), \(\bP_\cA\), \(h_\cA\) and \(\xi\)
define a Markov Young tower \(f \colon \Delta \to \Delta\).
To prove Theorem~\ref{thm:joke}, it remains to construct 
the semiconjugacy \(\pi \colon \Delta \to \Lambda\).

\subsection{Semiconjugacy}
\label{sec:semic}

Let \(\iota \colon Y \to \alpha^{\bN_0}\) be the natural embedding,
\(\iota(y) = (a_0, a_1, \ldots)\) if \(F^k(y) \in a_k\)
for all \(k\). (Technically, \(\iota\) is defined on a full measure subset of \(Y\).)
The space \(\alpha^{\bN_0}\) is supplied with the product topology and Borel sigma algebra.

\begin{remark}
  The map \(\iota\) is measurable and injective by construction;
  in addition we assumed that \(\iota(Y)\) is measurable in \(\alpha^{\bN_0}\).
  It is straightforward to check that \(\iota^{-1}\) is continuous on \(\iota(Y)\),
  and that \(\iota(A)\) is measurable for all measurable \(A \subset Y\).
  Hence \(\iota\) is \emph{bimeasurable:} both images and preimages of
  measurable sets are measurable.
\end{remark}

Let \(m_\alpha = \iota_* m\). This is a Borel probability measure on \(\alpha^{\bN_0}\)
with \(m_\alpha(\alpha^{\bN_0} \setminus \iota(Y)) = 0\).

For words \(w_0, \ldots, w_n \in \cA\), let
\(w_0 \cdots w_n\) denote their concatenation.
Then \(|w_0 \cdots w_n| = |w_0| + \cdots + |w_n|\)
and \(h_\cA(w_0 \cdots w_n) = h_\cA(w_0) + \cdots + h_\cA(w_n)\).

For \(x= (w_0, w_1, \ldots) \in X\), let
\(\pi_\alpha (x) \in \alpha^{\bN_0}\) denote the sequence of elements of
\(\alpha\) obtained by concatenating all \(w_k\), \(k \geq 0\).
It is clear that thus defined \(\pi_\alpha \colon X \to \alpha^{\bN_0}\)
is continuous.

\begin{proposition}
  \label{prop:sjhy}
  \((\pi_\alpha)_* \bP_X = m_\alpha\).
\end{proposition}

\begin{proof}
  Recall that we have the disintegration
  \( m = \sum_{w \in \cA} \bP_\cA(w) m_w \).

  Let \(w_0 \in \cA\). Since \(F^{|w_0|} \colon Y_{w_0} \to Y\) is
  a bijection and \(F^{|w_0|}_* m_{w_0} = m\), we can write
  \(m_{w_0} = \sum_{w_1 \in \cA} \bP_\cA(w_1) m_{w_0, w_1}\),
  where \(m_{w_0, w_1}\)
  are probability measures supported on \(Y_{w_0 w_1}\) such that
  \(F^{|w_0 w_1|}_* m_{w_0,w_1} = m\).
  Continuing with \(m_{w_0, w_1}\) and further recursively,
  we obtain for each \(n \geq 1\) a disintegration
  \[
    m = \sum_{w_0, \ldots, w_n \in \cA} \bP_\cA(w_0) \cdots \bP_\cA(w_n) m_{w_0, \ldots, w_n}
    ,
  \]
  where \(m_{w_0, \ldots, w_n}\) are probability measures supported on
  \(Y_{w_0 \cdots w_n}\) such that
  \(
    F^{|w_0 \cdots w_n|}_* m_{w_0, \ldots, w_n} = m
  \).
  
  Taking images under \(\iota \colon Y \to \alpha^{\bN_0}\),
  we obtain a similar disintegration in \(\alpha^{\bN_0}\):
  \[
    m_\alpha 
    = \sum_{w_0, \ldots, w_n \in \cA} \bP_\cA(w_0) \cdots \bP_\cA(w_n) m_{\alpha; \, w_0, \ldots, w_n}
    ,
  \]
  where \(m_{\alpha; \, w_0, \ldots, w_n} = \iota_* m_{\alpha; \, w_0, \ldots, w_n}\)
  are probability measures supported on the cylinders \(\alpha^{\bN_0}_{w_0 \cdots w_n}\)
  with
  \[
    \alpha^{\bN_0}_w
    = \{ (a_0, a_1, \ldots) \in \alpha^{\bN_0} : 
    a_0 \ldots a_{|w|-1} = w \}
    .
  \]

  Let \(w \in \cA\) and \(n = |w|\). Then
  for all \(w_0, \ldots, w_n \in \cA\), either
  \(\alpha^{\bN_0}_{w_0 \cdots w_n} \subset \alpha^{\bN_0}_w\)
  or \(\alpha^{\bN_0}_{w_0 \cdots w_n} \cap \alpha^{\bN_0}_w = \emptyset\).
  Thus
  \[
    m_\alpha(\alpha^{\bN_0}_w)
    = \sum_{\substack{w_0, \ldots, w_n \in \cA : \\ 
                      \alpha^{\bN_0}_{w_0 \cdots w_n} \subset \alpha^{\bN_0}_w 
           }} \bP_\cA(w_0) \cdots \bP_\cA(w_n)
    = \bP_X(\pi_\alpha^{-1} ( \alpha^{\bN_0}_w ) )
    .
  \]

  Thus \((\pi_\alpha)_* \bP_X\) agrees with \(m_\alpha\) on all cylinders in
  \(\alpha^{\bN_0}\). By Carath\'eodory's extension theorem,
  \((\pi_\alpha)_* \bP_X = m_\alpha\).
\end{proof}

Let \(\pi_X \colon X \to Y\), \(\pi_X = \iota^{-1} \circ \pi_\alpha\).

\begin{proposition}
  \label{prop:sjjs}
  \(\pi_X\) is well defined \(\bP_X\) almost everywhere on \(X\)
  and is measurable. Also, \((\pi_X)_* \bP_X = m\).
\end{proposition}

\begin{proof}
  The map \(\iota\) is injective, which allows us to define \(\pi_X\)
  on \(X' = (\pi_\alpha^{-1} \circ \iota )(Y)\).

  Recall that \(\iota(Y)\) is measurable in \(\alpha^{\bN_0}\)
  and \(m_\alpha(\iota(Y)) = 1\). The map \(\pi_\alpha\) is continuous, so
  \(X'\) is a measurable subset of \(X\) and,
  by Proposition~\ref{prop:sjhy}, \(\bP_X(X') = 1\).
  Hence \(\pi_X\) is defined almost everywhere.

  Using the bimeasurability of \(\iota\) and Proposition~\ref{prop:sjhy},
  for every measurable \(A \subset Y\),
  \[
    \bP_X(\pi_X^{-1}(A)) 
    = \bP_X( (\pi_\alpha^{-1} \circ \iota)(A) )
    = m_\alpha(\iota(A))
    = m(A)
    .
  \]
  In other words,
  \((\pi_X)_* \bP_X = m\).
\end{proof}

\begin{remark}
  Further we silently ignore the zero measure subset of \(X\), on which \(\pi_X\)
  is not defined, and the corresponding subset of \(\Delta\), which also has
  zero measure.
\end{remark}

Define \(\pi \colon \Delta \to \Lambda\) by
\begin{equation}
  \label{eq:mnr}
  \pi((w_0, w_1, \ldots), \ell) = T^\ell(\pi_X(w_0, w_1, \ldots)).
\end{equation}
Then \(\pi_X \colon \Delta_0 \to Y\) is a restriction of \(\pi\).

\begin{proposition}
  \(\pi\) is Lipschitz: for all \(a, b \in \Delta\),
  \[
    d_\Lambda(\pi(a), \pi(b)) \leq C_\Lambda \, d(a,b)
    ,
  \]
  where \(C_\Lambda = \lambda K_\tau \,\diam \Lambda \).
\end{proposition}

\begin{proof}
  Let \(a = (x_1, j)\) and \(b=(x_2,k)\), where
  \[
    x_1 = (w_{1,0}, w_{1,1}, \ldots)
    \quad \text{and} \quad
    x_2 = (w_{2,0}, w_{2,1}, \ldots)
    .
  \]
  If \(j \neq k\) or \(w_{1,0} \neq w_{2,0}\),
  then \(d(a, b) = 1\) and the statement is trivial.
  
  Suppose now that \(j = k\) and \(w_{1,0} = w_{2,0}\). 
  Let \(n = s(x_1, x_2)\). Note that \(n \geq 1\) and
  \[
    j = k 
    < h(x_1) = h(x_2) 
    = h_\cA(w_{1,0}) = h_\cA(w_{2,0})
    .
  \]
  Observe that \(\pi_X(x_i) \in Y_{w_{1,0} \cdots w_{1,n-1}}\) and
  \(F(\pi_X(x_i)) \in Y_{w_{1,1} \cdots w_{1,n-1}}\)
  for \(i=1,2\).
  Also, \(\diam Y_{w_{1,1} \cdots w_{1,n-1}} \leq \lambda^{-(n-1)} \, \diam Y\).
  Then
  \begin{align*}
    d_\Lambda(\pi(a),\pi(b))
    & = d_\Lambda \bigl( T^j(\pi_X(x_1)), T^j(\pi_X(x_2)) \bigr)
    \leq K_\tau \, \diam Y_{w_{1,1} \cdots w_{1,n-1}}
    \\ & \leq K_\tau \lambda^{-(n-1)} \, \diam Y
    = \lambda K_\tau \, \diam Y \, d(a,b).
  \end{align*}
\end{proof}

\begin{proposition}
  \label{prop:dnuy}
  \(T \circ \pi  = \pi \circ f \).
\end{proposition}

\begin{proof}
  Suppose that \(a = (x, \ell) \in \Delta\), and
  \(x = (w_0, w_1, \ldots)\).
   If \(\ell < h(x) - 1\), then
  \(f (a) = (x, \ell + 1)\) and
  \[
    \pi(f(a))
    = T^{\ell + 1} (\pi_X(x))
    = T (\pi(a))
    .
  \]
  If \(\ell = h(x) - 1\), then
  \[
    \pi(f(a))
    = \pi_X ( f_X (x) ) 
    = F(\pi_X(x)) 
    = T^{\ell + 1}(\pi_X(x))
    = T (\pi(a))
    .
  \]
  Thus
  \(
    \pi(f(a)) = T (\pi(a))
  \).
\end{proof}

\begin{proposition}
  \label{prop:ngbfg}
  \(\pi_* \bP = \nu\).
\end{proposition}

\begin{proof}
  We use the fact that \(\nu\) is the unique \(T\)-invariant ergodic
  probability measure on \(\Lambda\), with respect to which \(m\) is absolutely
  continuous.

  Since \(\bP\) is \(f\)-invariant and ergodic, it follows from Proposition~\ref{prop:dnuy}
  that \(\pi_* \bP\) is \(T\)-invariant and ergodic. 
  Since \(\bP_X\) is absolutely continuous with respect to \(\bP\)
  and \(\pi_* \bP_X = m\), using Proposition~\ref{prop:sjjs}
  we obtain that \(m\) is absolutely continuous with respect to \(\pi_* \bP\).
  Thus \(\pi_* \bP = \nu\).  
\end{proof}

\section{Proof of Theorem~\ref{thm:dura}}
\label{sec:dura}

\subsection{ASIP for Bernoulli shift}

Suppose that \(\{\eps_k\}_{k \in \bZ}\)
is a sequence of independent identically distributed
random variables, and \(X_k\) are real valued random variables
with mean zero given by
\[
  X_k = G(\ldots, \eps_{k-1}, \eps_k, \eps_{k+1}, \ldots)
\]
for some function \(G\).

Let \(\{\eps'_k\}\) be an independent copy of
\(\{\eps_k\}\) and for \(\ell \in \bZ\) define \(\{\eps_k^\ell\}_{k \in \bZ}\)
by
\[
  \eps_k^\ell = \begin{cases}
    \eps_k, & k \neq \ell, \\
    \eps'_\ell, & k = \ell.
  \end{cases}
\]
Define
\[
  X_k^\ell = G(\ldots, \eps_{k-1}^\ell, \eps_k^\ell, \eps_{k+1}^\ell, \ldots)
  .
\]
Let \(p > 4\),
\[
  \delta_{\ell,p}
  = \|X_0 - X_0^\ell\|_p
  \qquad \text{and} \qquad
  \Theta_{\ell,p} = \sum_{|k| \geq \ell} \delta_{k,p}
  ,
\]
where \(\| \cdot \|_p = \bigl(\bE |\cdot|^p \bigr)^{1/p}\).

We use the following result 
\cite[Theorem~2.1]{BLW14} (with \cite[Corollary~2.1]{BLW14}
to verify the assumptions):

\begin{theorem}
  \label{thm:BLW}
  If \(\|X_k\|_p < \infty\) and
  \(\Theta_{\ell,p} = o(\ell^{-p})\),
  then the partial sum process \(\sum_{k=0}^{n-1} X_k\)
  satisfies the ASIP with rate \(o(n^{1/p})\).
\end{theorem}

\begin{remark}
  Theorem~\ref{thm:BLW} is proved in \cite{BLW14} under a more relaxed
  condition on \(\Theta_{\ell,p}\). We use intentionally a suboptimal
  but easy to state result.
\end{remark}

\subsection{Construction of Bernoulli shift}
\label{sec:cbs}

Suppose that \(f \colon \Delta \to \Delta\) is a Markov Young tower
as in Section~\ref{sec:BYT} with 
\[
  \bP_\cA(h_\cA = n) = 
  \begin{cases}
    \theta (1-\theta)^{-1} (1-\theta)^{n/N}, & n \in \{N, 2N, 3N, \dots\} \\
    0, & \text{else}
  \end{cases}
\]
with \(N \in \bN\) and \(0 < \theta < 1\).
Let \(v \colon \Delta \to \bR\) be a centered H\"older observable
and \(v_n = \sum_{k=0}^{n-1} v \circ f^k\) be the corresponding
random process on \((\Delta, \bP)\).

By Theorem~\ref{thm:joke}, to prove Theorem~\ref{thm:dura}
it is enough to show the ASIP for \(v_n\).

The map \(f\) is \(N\)-periodic. For simplicity we assume that \(N = 1\).
We show how to remove this assumption in Subsection~\ref{sec:N}.

In the rest of this subsection we construct a suitable Bernoulli shift
\(\sigma \colon D \to D\) with a measure preserving
semiconjugacy \(g \colon D \to \Delta\). 
The random process \(\sum_{k=0}^{n-1} X_k\) with
\(X_k = v \circ g \circ \sigma^k\)
has the same distribution as \(v_n\).
If \(\{\eps_k\}\) are the coordinates of \(D\), they are
independent and identically distributed, and
\(X_k = (v \circ g) (\ldots, \eps_{k-1}, \eps_k, \eps_{k+1}, \ldots)\).
This sets up a ground for the application of Theorem~\ref{thm:BLW}.

Let \((\Omega, \bP_\Omega)\) be a probability space supporting 
random variables \(A_n \colon \Omega \to \cA\), \(n \geq 1\),
such that for \(a \in \cA\),
\[
  \bP(A_n = a) = 
  \begin{cases}
    0, & h_\cA(a) \neq n, \\
    \frac{\bP_\cA(a)}{\bP_\cA(\bigcup \{ a \in \cA : h_\cA(a) = n\})},
    & h_\cA(a) = n
  \end{cases}
  .
\]
That is, \(A_n\) is a random element of \(\cA\) chosen among those with
\(h_\cA = n\) with respect to the appropriately conditioned measure \(\bP_\cA\).

Let \(Z = \{0,1\}\) and 
\(\bP_Z\) be the probability measure on \(Z\) given by 
\(\bP_Z(0) = 1-\theta\) and \(\bP_Z(1) = \theta\).

Define \(D = (\Omega \times Z)^\bZ\) with the product probability measure
\(\bP_D = (\bP_\Omega \times \bP_Z)^\bZ\).
Let \(\eps_k = (\omega_k, z_k)\) be the coordinates in \(D\) and
\(\sigma \colon D \to D\) be the left shift.

Let 
\[
  t_0 = \sup \{k \leq 0 : z_k = 1\}
  \qquad \text{and} \qquad
  t_n = \inf \{k > t_{n-1} : z_{k} = 1\},
  \quad n \geq 1.
\]
Note that \(t_n\) are finite \(\bP_D\)-almost surely.

Define \(g \colon D \to \Delta\) by 
\( g(\{\eps_k\}) = (y, -t_0)\), where
\(
  y =
  (
    A_{t_1 - t_0}(\omega_{t_0}), 
    A_{t_2 - t_1}(\omega_{t_1}),
    \ldots
  )
\).
Observe that \(g\) is a probability measure preserving semiconjugacy between
\(\sigma \colon D \to D\) and \(T \colon \Delta \to \Delta\).

\subsection{Weak dependence}

Here we verify the assumptions of Theorem~\ref{thm:BLW}.
As above, we set
\[
  X_k = (v \circ g) (\ldots, \eps_{k-1}, \eps_k, \eps_{k+1}, \ldots)
  .
\]
Let \(p > 4\). The observable \(v\) is H\"older continuous,
thus \(\|X_k\|_p \leq \|X_k\|_\infty < \infty\).
It remains to prove that \(\Theta_{\ell,p} = o(\ell^{-p})\).

\begin{proposition}
  \label{prop:najji}
  There exists \(0 < \theta_\delta < 1\) such that
  \(\delta_{\ell,p} = O\bigl(\theta_\delta^{|\ell|}\bigr)\).
\end{proposition}

\begin{proof}
  Let \(\ell \in \bZ\),
  \[
    x = ((w_0, w_1, \ldots), r) = g(\{\eps_k\})
    \quad \text{and} \quad 
    x^\ell = ((w_0^\ell, w_1^\ell, \ldots), r^\ell) = g(\{\eps_k^\ell\})
    .
  \]

  Suppose first that \(\ell \geq 1\). Let \(c_\ell = \sum_{j=1}^{\ell-1} z_j\).
  Then \(w_j = w_j^\ell\) for all \(0 \leq j \leq c_\ell-1\).
  If \(c_\ell \geq 1\), then \(r = r^\ell\)
  and \(d(x,x^\ell) \leq \xi^{c_\ell}\).
  If \(c_\ell = 0\), then \(d(x, x^\ell) \leq 1\).
  In either case, 
  \[
    d(x,x^\ell) \leq \xi^{c_\ell}
    .
  \]

  Since \(\{z_k\}\) are independent identically distributed,
  \[
    \bE d(x, x^\ell)^p
    \leq \bE \xi^{p c_\ell}
    = \bigl( \bE \xi^{p z_3}\bigr)^{\ell-1}
    = \bigl( 1-\theta + \theta \xi^p \bigr)^{\ell-1}
    .
  \]
  Since \(v\) is H\"older continuous,
  \(|X_0 - X_0^\ell| \leq |v|_\eta d(x,x^\ell)^\eta\)
  and the result for \(\ell \geq 1\) follows.
  
  Suppose now that \(\ell \leq 0\). Then \(x \neq x^\ell\)
  only when \(t_0 \leq \ell\).
  The result follows from H\"older continuity of \(v\) and
  \[
    \bP(t_0 \leq \ell) 
    = \bP(z_0 = z_{-1} = \cdots = z_{\ell-1} = 0) 
    = (1-\theta)^\ell.
  \]
\end{proof}

Finally, \(\Theta_{\ell,p}\) decays exponentially in \(\ell\),
because so does \(\delta_{\ell,p}\).
The proof of Theorem~\ref{thm:dura} is complete.

\subsection{Periodic tower}
\label{sec:N}

In Subsection~\ref{sec:cbs} we assumed that the tower \(f \colon \Delta \to \Delta\)
is aperiodic, namely that \(N = 1\). Here we give a sketch of proof for \(N > 1\).

Let
\[
  \Delta_N = \{(x,\ell) \in \Delta : \ell = 0 \pmod N \}
  .
\]
We supply \(\Delta_N\) with a probability measure \(\bP_N\), which
is a (normalized) restriction of \(\bP\).
Define a projection \(\pi_N \colon \Delta \to \Delta_N\),
\(\pi_N (x, \ell) = \big(x, N \big\lfloor \frac{\ell}{N}\big\rfloor\big)\).
Observe that \((\pi_N)_* \bP = \bP_N\).

Let 
\[
  u_n = \sum_{k=0}^{n-1} u \circ f^{kN}
  \qquad \text{with} \qquad
  u = \sum_{k=0}^{N-1} v \circ f^k
\]
be a process on the probability space \((\Delta_N, \bP_N)\).

Since \(\pi_N\) is measure preserving and 
\(|v_n - u_{\lfloor n/N \rfloor} \circ \pi_N| \leq 2 N |v|_\infty \),
the processes \(v_n\) and \(u_n\) are naturally defined on a common probability
space with \(v_n = u_{\lfloor n/N \rfloor} + O(1)\) almost surely.

By the method which works for \(N = 1\),
we show the ASIP for the process \(u_n\) on
the probability space \((\Delta_N, \bP_N)\).
So, there exists a Brownian motion \(W_n\) with
\(u_n = W_n + o(n^\eps)\) almost surely for every \(\eps > 0\).

Thus, almost surely and for every \(\eps > 0\),
\[
  v_n
  = W_{\lfloor n/N \rfloor} + o(n^\eps)
  = W'_{n} + o(n^\eps)
  ,
\]
where \(W'_n = W_{n/N}\) is a Brownian motion.
We used that \(W_{n/N} - W_{\lfloor n/N \rfloor} = O(\log n)\)
almost surely.
This is the desired ASIP for \(v_n\).

\appendix

\section{Proof of Lemma~\ref{lemma:ccnnjj}}
\label{sec:kkk}

Our argument is based on \cite[Section 4]{KKM16exp}, and here
we work in their notations, which are different from the rest of
the paper.

In this section, \(T \colon \Lambda \to \Lambda\) is a nonuniformly
expanding map as in Section~\ref{sec:YT}, 
\(F \colon Y \to Y\) is the induced map
and \(f \colon \Delta \to \Delta\) is
the Young tower,
\begin{align*}
  \Delta 
  & = \{ (y,\ell) \in Y \times \bZ : 0 \leq \ell < \tau(y) \}
  ,
  \\ f(y,\ell) 
  & = \begin{cases} (y, \ell+1), & \ell < \tau(y)-1, \\ (Fy, 0), & \ell=\tau(y)-1. \end{cases}
\end{align*}
Let \(\bar{\tau} = \int_Y \tau \, dm\).
Let \(m_\Delta\) be the probability measure on \(\Delta\) given by
\(
  m_\Delta(A \times \{\ell\}) = \bar{\tau}^{-1} m(A)
\)
for all \(\ell \geq 0\) and $A \subset \{y\in Y:\tau(y)\ge \ell+1\}$.

Let \(L \colon L^1(m_\Delta) \to L^1(m_\Delta)\) be the transfer operator corresponding
to \(f\) and \(m_\Delta\), so \(\int L \phi \, \psi \, dm_\Delta = \int \phi \, \psi \circ f \, dm\)
for all \(\phi \in L^1\) and \(\psi \in L^\infty\).

Without loss of generality we assume that \(f\) is mixing
(otherwise we switch to a power of \(f\)
which is mixing).

\begin{remark}
  The proof of decay of correlations in \cite{KKM16exp} is based on a construction of
  a probability space \((\bW, \bP_\bW)\) and a random variable \(r \colon \bW \to \bN\) such that
  each sufficiently regular observable \(\psi \colon \Delta \to [0,\infty)\)
  with \(\int \psi \, dm_\Delta = 1\) can be decomposed into a sum
  \(\psi = \sum_{w \in \bW} \bP_\bW(w) \psi_w\) with \(\int \psi_w \, dm_\Delta = 1\)
  and \(L^{r(w)} \psi_w = \bar{\tau} 1_{\Delta_0}\).
  In particular, this applies to \(\psi = \bar{\tau} 1_{\Delta_0}\).

  The distribution of \(r\) depends on the tails 
  \(m(\tau > n)\). There is quite a lot of flexibility
  in the construction. We show that if the tails decay exponentially,
  we can construct \(r\) distributed geometrically up to a period,
  as required for Lemma~\ref{lemma:ccnnjj}.
  Moreover, we can take \(\bW = \cA\), \(r = h_\cA\)
  and ensure that the observables \(\psi_w\) are supported on the
  respective \(Y_w \times \{0\}\).
  This yields the desired result, with \(m_w\) given by the densities \(\psi_w\).
\end{remark}

\begin{remark}
  Unfortunately, there is no easy way to point out what needs to be changed
  in \cite{KKM16exp}. We present a complete proof, referring to \cite{KKM16exp}
  in the proofs where possible.
\end{remark}

Let \(\Delta_\ell = \{(y,k) \in \Delta : k = \ell\}\).
Recall that \(\eta \in (0,1]\) is the exponent in~\eqref{eq:dist}.
For \(\psi \colon \Delta \to [0,\infty)\), define
\[
  |\psi|_{\eta,\ell} = \sup_{n\geq 0} \sup_{(y,n)\neq (y',n) \in \Delta_n} 
  \frac{|\log \psi(y,n)- \log \psi(y',n)|}{d(y,y')^\eta},
\]
where \(\log 0 = -\infty\) and \(\log 0 - \log 0 = 0 \).
Note that for a countable collection \(\psi_k\) of nonnegative functions,
\(\bigl|\sum_k \psi_k\bigr|_{\eta, \ell} \leq \max_k |\psi_k|_{\eta, \ell}\).

For \(a \in \alpha\), let
\(
  S_a
  = \{ (y, k) \in \Delta : y \in a \text{ and } k = \tau(y)-1 \}
\),
and let \(\varkappa\) be the partition of \(\Delta\) generated by \(\{S_a\}_{a \in \alpha}\)
and \(\{\Delta_\ell\}_{\ell \geq 0}\). Let \(\varkappa^n = \vee_{k=0}^{n-1} f^{-k} \varkappa\).
Then \(\varkappa^0\) is the trivial partition, and for every \(n \geq 1\) and \(a \in \varkappa^n\),
there exists \(\ell \geq 0\) such that \(f^n \colon a \to \Delta_\ell\) is a bijection.

Fix constants $R>0$ and $\xi \in (0,e^{-R})$ such that
$ R (1-\xi e^R) \geq K + \lambda^{-1} R$.

\begin{proposition}
  \label{prop-duude}
  Suppose that $\psi:\Delta\to[0,\infty)$ with \(|\psi|_{\eta, \ell} \leq R\).
  Let \(n \geq 1\), \(a \in \varkappa^n\) and \(\psi_a = \psi 1_a\). Then
  \begin{enumerate}[label=(\alph*)]
    \item
      \(
        e^{-R} \bar{\tau} \int_{\Delta_0} \psi \, dm_\Delta  
        \leq \psi \, 1_{\Delta_0} 
        \leq e^{ R} \bar{\tau} \int_{\Delta_0} \psi \, dm_\Delta .
      \)
    \item \(|L^n \psi_a|_{\eta, \ell} \leq R\).
    \item If \(t \in [0, \xi]\), then
      \(
        \psi_a' =
        L^n \psi_a - t \, \bar{\tau} \int_{\Delta_0} L^n \psi_a \, dm_\Delta \, 1_{\Delta_0}
      \)
      is nonnegative and \(|\psi_a'|_{\eta, \ell} \leq R\).
  \end{enumerate}
\end{proposition}

\begin{proof}
  This is a minor modification of \cite[Proposition~4.1]{KKM16exp}.
\end{proof}

Let \(\cR\) be the set of observables \(\psi : \Delta \to [0,\infty)\)
such that 
\(|\psi|_\infty \leq e^R \bar{\tau} \int_\Delta \psi \, dm_\Delta\) and
\(|\psi|_{\eta, \ell} \leq R\).

For \(n \geq 0\), let \(\cR^n\) denote the set of observables
\(\psi \colon \Delta \to [0,\infty)\) such that \(L^n \psi \in \cR\) and
\(|L^n (\psi 1_a)|_{\eta,\ell} \leq R\) for every \(a \in \varkappa^n\).

\begin{corollary} \label{cor-A} \ 
  \begin{itemize}
    \item[(a)] If \(\psi \colon \Delta \to [0,\infty)\) is supported on
      \(\Delta_0\) and \(|\psi|_{\eta, \ell} \leq R\), then
      \(\psi \in \cR\).
    \item[(b)] If \(\psi \in \cR\), then \(L \psi \in \cR\).
    \item[(c)] If \(\psi \in \cR^n\), then \(\psi \in \cR^k\) for all \(k \geq n\).
    \item[(d)] If \(\psi, \psi' \in \cR^n\) and \(t \geq 0\), then
      \(\psi + \psi'\) and \(t \psi\) belong in \(\cR^n\).
  \end{itemize}
\end{corollary}

\begin{proof}
  See \cite[Corollary~4.2]{KKM16exp}.
\end{proof}

\begin{lemma}
  \label{lem-recur}
  There exist \(N \geq 1\) and \(\eps > 0\) such that
  \begin{enumerate}[label=(\alph*)]
    \item
      \(
        \int_{\Delta_0} \psi \, dm_\Delta 
        \geq \eps \int_\Delta \psi \, dm_\Delta
      \)
      for all \(\psi \in L^N \cR\),
    \item
      \(
        (1-\xi \eps) \bigl( \frac{1-\xi}{1-\xi \eps} \bigr)^{n}
        \geq e^R \bar{\tau} m_\Delta \bigl( \bigcup_{\ell=Nn}^\infty \Delta_\ell \bigr)
      \)
      for all \(n \geq 1\).
  \end{enumerate}
\end{lemma}

\begin{proof}
  (a) is proved in \cite[Lemma~4.5]{KKM16exp}. Following the
  proof, we are free to choose \(\eps\) as small as needed
  and \(N\) as large as needed.
  By assumptions of Lemma~\ref{lemma:ccnnjj}, 
  \(m_\Delta \bigl( \bigcup_{\ell=n}^\infty \Delta_\ell )\) decays exponentially
  in \(n\), thus we can choose \(N\) and \(\eps\) so that (b) is satisfied.
\end{proof}

Further we assume that \(N\) and \(\eps\) are as in Lemma~\ref{lem-recur}.
Define \(\cB = L^N \cR\). Note that \(L \cB \subset \cB \subset \cR\).
For \(n \geq 0\) let \(\cB^n\) denote the set of observables \(\psi \colon \Delta \to [0,\infty)\)
such that \(L^n \psi \in \cB\) and \(|L^n (\psi 1_a)|_{\eta, \ell} \leq R\) 
for every \(a \in \varkappa^n\).

\begin{remark}
  If \(\psi \in \cB\), then \(L \psi \in \cB\).
  If \(\psi \in \cB^n\), then \(\psi \in \cB^k\) for all \(k \geq n\).
\end{remark}

Define a sequence \(p_n, n \geq -1\) by
\[
  p_{-1} = \xi \eps
  \quad \text{and} \quad
  p_n = 
  \begin{cases}
    (1-\xi) \eps \bigl( \frac{1-\eps}{1-\xi \eps} \bigr)^{n/N}
    , & n \in N \bZ \\
    0, & n \not \in N \bZ
  \end{cases}
  \quad \text{for } n \geq 0
  .
\]
Let \(t_n = 1 - \sum_{k=-1}^{n-1} p_k\) for \(n \geq 1\).
Then \(\sum_{k=-1}^{\infty} p_k = 1\), \(t_1 = 1-\eps\) and
for \(n \geq 2\) using Lemma~\ref{lem-recur} we obtain
\beq \label{eq:maae} \begin{aligned}
  t_{Nn} 
  = 1-\sum_{k=-1}^{Nn-1} p_k
  = (1-\xi \eps) \Bigl( \frac{1-\eps}{1-\xi \eps} \Bigr)^n
  \geq \min \bigl\{t_1,\, e^R \bar{\tau} 
  m_\Delta \bigl(\textstyle\bigcup_{\ell=Nn}^\infty \Delta_\ell \bigr) \bigr\}
  .
\end{aligned} \eeq

Let \(E_0 = \Delta_0\) and
\(E_k = \{ (y, \ell) \in \Delta : \ell=\tau(y)-k,\,\ell\ge1\}\)
for \(k \geq 1\). Then $\{E_0,E_1,\ldots\}$ defines a partition of $\Delta$ and 
\(m_\Delta (E_k) = m_\Delta(\Delta_k)\) for all \(k\).

\begin{proposition}
  \label{prop-E}
  If \(\psi \in \cB\) with $\int_\Delta \psi \, dm_\Delta=1$, then
  \(
    \int_{\bigcup_{\ell=n}^\infty E_\ell} \psi \, dm_\Delta
    \leq t_n,
  \)
for $n\ge1$.
\end{proposition}

\begin{proof}
  See \cite[Proposition~4.6]{KKM16exp}.
\end{proof}

\begin{proposition}  \label{prop-s}
  Let $p_j$, $q_j\in[0,\infty)$ be sequences such that
  \(\sum_{j=0}^{\infty} p_j = \sum_{j=0}^{\infty} q_j < \infty\)
  and $\sum_{j=0}^k q_j\ge \sum_{j=0}^k p_j$ for all $k\ge0$.
  Then there exist $s_{k,j}\in[0,1]$, $0\le j\le k$, such that
  $\sum_{j=0}^k s_{k,j}q_j=p_k$ for all $k\ge0$ and
  $\sum_{k=j}^\infty s_{k,j}=1$ for all $j\ge0$.
\end{proposition}

\begin{proof}
  See \cite[Proposition~4.7]{KKM16exp}.
\end{proof}

\begin{lemma}
  \label{lem-W}
  Let \(\psi \in \cB^n\) for some \(n \geq 0\).
  Then 
    $\psi =  \sum_{k=-1}^{\infty} \psi_k$,
  where \(\psi_k : \Delta \to [0,\infty)\) are such that
  \begin{itemize}
    \item[(a)] \(L^n (\psi_{-1} 1_a) = c_a 1_{\Delta_0}\)
      for all \(a \in \varkappa^n\), where \(c_a\) are nonnegative constants,
    \item[(b)] \(\sum_{a \in \varkappa^n} c_a 
      = p_{-1} \bar{\tau} \int_\Delta \psi \, dm_\Delta\),
    \item[(c)]
      \(\psi_k \in \cR^{n+k}\) for all $k\ge0$,
    \item[(d)] \(\int_\Delta \psi_k \, dm_\Delta = p_k \int_\Delta \psi \, dm_\Delta\)
      for all \(k \geq -1\)
  \end{itemize}
\end{lemma}

\begin{proof}
  We follow the proof of \cite[Lemma~4.8]{KKM16exp}.
  Suppose without loss of generality that \(\int_\Delta \psi \, dm_\Delta = 1\).
    
  Define
  \[
    \textstyle 
    t = p_{-1} / \int_{\Delta_0}L^n \psi\,dm_\Delta
    = \xi\eps/\int_{\Delta_0}L^n \psi\,dm_\Delta
    .
  \]
  By Lemma~\ref{lem-recur},
  $\int_{\Delta_0} L^n \psi\,dm_\Delta\ge\eps$,
  so \(t \in [0,\xi]\).
  
  Under convention that \(0/0 = 0\), let
  \[
    \psi_{-1} = t \bar{\tau} \sum_{a \in \varkappa^n} \left(
      \frac{\int_{\Delta_0} L^n (\psi 1_a) \, dm_\Delta}{L^n (\psi 1_a)}
      \circ f^n
    \right) \psi 1_a
    .
  \]
  Then properties (a) and (b) are satisfied.
  
  Let \(g = \psi - \psi_{-1}\) and \(g_k = g 1_{T^{-n} E_k}\) for \(k \geq 0\).
  Then \(L^{n+k} g_k\) is supported on \(\Delta_0\) and
  \(|L^{n+k} (g_k 1_a)|_{\eta, \ell} \leq R\) for every \(a \in \varkappa^n\).
  By Corollary~\ref{cor-A}, \(g_k \in \cR^{n+k}\).
  
  Let \(q_k = \int_\Delta g_k \, dm_\Delta\). Then 
  \(\sum_{k=0}^\infty q_k = \sum_{k=0}^\infty p_k\) and
  by Proposition~\ref{prop-E},
  \(\sum_{k=0}^n q_k \geq \sum_{k=0}^n p_k\) for all \(n \geq 0\).
  Choose  \(s_{k,j} \in [0, 1]\) as in
  Proposition~\ref{prop-s}, and 
  define \(\psi_k : \Delta \to [0,\infty)\), $k\ge0$, by
  \[
    \textstyle
    \psi_k=\sum_{j=0}^k s_{k,j}g_j.
  \]
  Then (d) holds for all \(k\).
  Corollary~\ref{cor-A} implies (c).
\end{proof}

Let \(\bW\) be the countable set of all finite words in the alphabet 
$\bN_0$ including the zero length word, 
and let \(\bW_k\) be the subset consisting of words of length \(k\). 
Let \(\bP_\bW\) be the  probability measure
on \(\bW\) given for \(w = w_1 \cdots w_k \in \bW_k\) by
\(\bP_\bW(w) = p_{-1} p_{w_1} \cdots p_{w_k}\).
Define $r\colon \bW\to \bN_0$ by
$r(w)=\Sigma w+N|w|$, where
$\Sigma w=w_1+\dots+ w_k$ and $|w|=k$ for $w=w_1\cdots w_k$.

\begin{proposition}
  \label{prop-W}
  Let \(\psi \in \cB\) with \(\int_\Delta \psi \, dm_\Delta = 1\).
  Then
    $\psi = \sum_{w \in \bW} \psi_w$,
  where \(\psi_w : \Delta \to [0, \infty)\) are
  such that 
  \begin{itemize}
    \item[(a)] \(\int_\Delta \psi_w \, dm_\Delta = \bP_\bW(w)\),
    \item[(b)] \(L^{r(w)} \psi_w = \bP_\bW(w) \bar{\tau} 1_{\Delta_0}\),
    \item[(c)] \(L^{r(w)}(\psi_w 1_a) = c_{w,a} 1_{\Delta_0}\)
      for all \(a \in \varkappa^{r(w)}\), where \(c_{w,a}\)
      are nonnegative constants.
  \end{itemize}
\end{proposition}

\begin{proof}
  Proof is identical to \cite[Proposition~4.9]{KKM16exp}
  except for condition~(c), which is guaranteed by
  Lemma~\ref{lem-W}.
\end{proof}

\begin{definition}
  We say that a random variable \(X\) has geometric distribution
  with parameter \(\theta \in (0,1)\)
  (or \(X \sim \Geom(\theta)\)),
  if \(X\) takes values in \(\bN_0\) and
  \(\bP(X = n) = (1-\theta)^n \theta\) for \(n \geq 0\).
\end{definition}

\begin{proposition}
  \label{prop:geoi}
  Suppose that \(Y = \sum_{k=1}^{M} (1+X_k)\), where \(M \sim \Geom(\theta_M)\)
  and \(X_k \sim \Geom(\theta_X)\) are independent random variables.
  Let \(\eta_2 = \theta_X \theta_M\) and
  \(\eta_1 = \frac{\theta_M-\eta_2}{1-\eta_2}\).
  Then
  \[
    \bP(Y = n) = 
    \begin{cases}
      \eta_1 + (1-\eta_1) \eta_2, & n = 0 \\
      (1-\eta_1) \eta_2 (1-\eta_2)^n, & n \geq 1
    \end{cases}
    .
  \]
\end{proposition}

\begin{proof}
  We compute the probability generating function of \(Y\). For \(z \in \bR\),
  \[
    \bE(z^Y) = \bP(M = 0) + \bP(M \geq 1) \bE(z^{1+X_1}) \bE(z^Y)
    .
  \]
  Using
  \[
    \textstyle
    \bE(z^{1+X_1})
    = \sum_{k=0}^{\infty} \bP(X_1 = k) z^{k+1}
    = \frac{\theta_X z}{1-(1-\theta_X)z}
    ,
  \]
  we obtain
  \[
    \textstyle
    \bE(z^Y) = \eta_1 + (1-\eta_1) \frac{\eta_2}{1-(1-\eta_2) z}
    .
  \]
  Now, \(\bP(Y=n)\) is the coefficient at \(z^n\) in the above expression.
\end{proof}

\begin{proposition}
  \label{prop:njhal}
  There exist constants \(0 < \theta < 1\) and \(C_1, C_2 > 0\) such that
  \[
    \bP(r = n)
    = \begin{cases}
      C_1, &  n = 0 \\
      C_2 \theta^{n/N}, & n \in N \bN \\
      0, & \text{else}
    \end{cases}
  \]
\end{proposition}

\begin{proof}
  Recall that \(\bW_k\) is the subset of \(\bW\) consisting of words of length \(k\).
  Then \(\bP_\bW(\bW_k) = (1-p_{-1})^k p_{-1}\).
  Elements of \(\bW_k\) have the form \(w_1 \cdots w_k\)
  where \(w_1, \ldots, w_k\) can be regarded as independent identically
  distributed random variables, drawn from $\bN_0$ with distribution
  \[
    \bP(w_1 = n) = p_n/(1-p_{-1})
    = \begin{cases}
      \theta_1 (1-\theta_1)^{n/N}, & n \in N \bN_0 \\
      0 , & \text{else}
    \end{cases}
    ,
  \]
  where \(\theta_1 = \frac{(1-\xi)\eps}{1-\xi \eps}\). In other words,
  \(w_1/N \sim \Geom(\theta_1)\).
  
  Then the random variable \(r/N\) on \(\bW\) has the same distribution
  as \(Y\) in Proposition~\ref{prop:geoi} with \(\theta_M = p_{-1}\)
  and \(\theta_X = \theta_1\). The result follows.
\end{proof}

We are ready to complete the proof of Lemma~\ref{lemma:ccnnjj}.
Let \(\psi = dm / dm_\Delta =  \bar{\tau} 1_{\Delta_0}\) and
\(\psi = \sum_{w \in \bW} \psi_w\) be the decomposition
from Proposition~\ref{prop-W}.

Then
\(
  \psi 
  = \sum_{w \in \bW} \sum_{a \in A(w)} \psi_w 1_a
\),
where \(A(w) = \{a \in \varkappa^{r(w)} \colon a \subset \Delta_0 \text{ and } 
f^{r(w)}a = \Delta_0\}\).
To every \(w \in \bW\) and \(a \in A(w)\) there corresponds
\(u \in \cA\) such that \(a = Y_{u}\) (modulo zero \(m\) measure)
and \(r(w) = h_\cA(u)\).
Thus we can write
\[
  m = \sum_{u \in \cA} \bP_\cA(u) m_{u},
\]
where \(m_{u}\) are probability measures supported on \(Y_{u}\) and
\(\bP_\cA\) is a probability measure on \(\cA\)
such that \(\bP_\cA(h_\cA=n) = \bP_\bW(r=n)\) for all \(n\).

The result of Lemma~\ref{lemma:ccnnjj} follows from Proposition~\ref{prop:njhal}.

\subsection*{Acknowledgments}
This research was supported in part by a European
Advanced Grant {\em StochExtHomog} (ERC AdG 320977)
at the University of Warwick
and an Engineering and Physical Sciences Research Council grant
EP/P034489/1 at the University of Exeter.

The author is grateful to Mark Holland and Ian Melbourne for support.
The author has been very lucky with the referees who
were aware of the history of the problem 
and made invaluable suggestions for improvement of the manuscript.

The author is thankful to Christophe Cuny, J\'er\^ome Dedecker
and Florence Merlev\`ede for helpful comments and for pointing out 
a mistake in the construction of semiconjugacy.

\end{document}